\documentclass[12pt]{article}
\usepackage{}
\usepackage{stmaryrd}
\usepackage{amssymb}

\usepackage{graphics,amsmath,amssymb,amsthm,mathrsfs}

\newcommand{\rn}[1]{\mathbb{R}^{#1}}

\newtheorem{thm}{Theorem}[section]
\newtheorem{lemma}[thm]{Lemma}

\newtheorem{remark}[]{Remark}

\newcounter{RomanNumber}

\numberwithin{equation}{section}

\begin{document}
\bibliographystyle{amsplain}
\title{Concentration Behavior of Nonlinear Hartree-type Equation with almost Mass Critical Exponent \\ }
\author{Yuan Li{$^a$} Dun Zhao{$^a$}  Qingxuan Wang$^{b}$\footnote{Corresponding author: this work is supported by the NSFC Grants 11801519 and 11475073.}
\renewcommand\thefootnote{}
\footnote{{E-mail addresses}: liyuan2014@lzu.edu.cn (Y. Li), zhaod@lzu.edu.cn (D. Zhao), wangqx@zjnu.edu.cn (Q. Wang).}
\setcounter{footnote}{0}
 \\\small ${^a}$ School of Mathematics and Statistics, Lanzhou
 University, Lanzhou, 730000, PR China
 \\\small ${^b}$ Department of Mathematics, Zhejiang Normal
 University,\\ \small Jinhua, 321004, Zhejiang, PR China}
\date{ }
\maketitle
\begin{abstract}
We study the following nonlinear  Hartree-type  equation
\begin{equation*}
-\Delta u+V(x)u-a(\frac{1}{|x|^\gamma}\ast |u|^2)u=\lambda u,~\text{in}~\mathbb{R}^N,
\end{equation*}
where $a>0$, $N\geq3$, $\gamma\in(0,2)$ and $V(x)$ is an external potential. We first study the asymptotic behavior of the ground state of equation for $V(x)\equiv1$, $a=1$ and $\lambda=0$ as $\gamma\nearrow2$.  Then we consider the case of some trapping  potential $V(x)$, and show that all the mass of ground states concentrate at a global minimum point of $V(x)$  as $\gamma\nearrow2$, which leads to symmetry breaking. Moreover, the concentration rate for  maximum points of ground states will be given.\vspace{0.2cm}
\end{abstract}

\section{Introduction}\label{section:1}

In this paper, we consider the concentration behavior of the following nonlinear Hartree-type equation as $\gamma \nearrow 2$
\begin{equation}\label{eq1}
-\Delta u+V(x)u-a(\frac{1}{|x|^\gamma}\ast |u|^2)u=\lambda u,~~\text{in}~\mathbb{R}^N,
\end{equation}
 where $\gamma>0$, $N\geq3$, $a>0$ is the parameter, $\lambda\in\mathbb{R}$,  and $V(x)$ is an external potential. This equation can be used to describe the standing waves of the form $\psi(t,x)=e^{i\lambda t}u(x)$ of the focusing time-dependent equation
\begin{equation}\label{time-equ}
i\psi_t=-\Delta\psi+V(x)\psi-a(\frac{1}{|x|^\gamma}*|\psi|^2)\psi,
~~\text{in}~~\rn{N}\times\mathbb{R}_+.
\end{equation}

 Equation (\ref{time-equ}) can  describe the geometry of stars and planets in celestial mechanics \cite{L1933}, and quantum mechanics for investigating Bose-Einstein condensates (BEC) and Thomas-Fermi type problems \cite{DGPS1999}. In particular, for $a>0$ and $\gamma=1$, the  interaction is attractive Coulomb action, the model can describe the quantum mechanics of a polaron, see \cite{L1977,P1954}.

There have been a great deal of papers devote to equation \eqref{eq1} in different aspects. For results about the existence of positive ground states, one can see \cite{CSV2008,LZ2010,LXZ2016,MPT1998,MS2013,MS2015,Li-Ye2014}; Cingolani \cite{Cing-Secch2012} and Clap \cite{Clap-Sal-multi2013} investigated the existence of multiple solutions; for the uniqueness of the ground state, see for example \cite{LZ2010,L1976/1977,S,WY2017,X2016}; and the semi-classical analysis results, see \cite{Moroz-Schaf2015,Cing-Secchi2010} and the references therein.

 When $\gamma=2$, the above Hartree-type nonlinearity in $\mathbb{R}^N$ is corresponding to mass critical case for all $N\geq3$. The authors in \cite{DL2015,LXZ2016,Ye2016} considered a minimizing variational problem corresponding to (\ref{eq1}), and  proved that there exists a constant $a^*$ such that (\ref{min}) admits at least one minimizer if and only if $a<a^*$, where $a^*=\|Q_2\|^2_2,$
and $Q_2$ is the positive radially symmetric ground state of the following  equation
\begin{equation}\label{equ-gamma=2}
-\Delta u+u-(\frac{1}{|x|^2}\ast |u|^2)u=0,\ \ \text{in}~\mathbb{R}^N.
\end{equation}
Furthermore, the concentration and symmetry breaking of minimizers as $a\nearrow a^*$ were also obtained for different potentials  $V(x)$. We mention that such kind of mass critical problems have been studied for cubic nonlinearity in $\mathbb{R}^2$ in the past few years, see \cite{GS2014,Guo_zheng_zhou2015,Deng_Guo_Lu2015,GZ2014,WZ2017} and the references therein.

Here we state the work of Guo, Zeng and Zhou \cite{GZ2014}, in which they studied concentration behavior of the following almost mass critical nonlinear Schr\"{o}dinger equations as $q\nearrow 2$ ($2$ is mass critical exponent)
\begin{align*}
-\triangle u + (V(x)+\lambda)u-a |u|^q u=0, \ \ \ \ (t,x)\in \mathbb{R}^1\times\mathbb{R}^2,
\end{align*}
where $a>a^*=\|Q\|^2_2$, $Q$ is the unique (up to translations) radially symmetric positive solution of the following limiting equation
\begin{align}
-\triangle u + u- |u|^2 u=0,\ \ \ \text{in $\mathbb{R}^2$}.
\end{align}
By using constrained variational method and energy estimates they present a detailed analysis of the concentration and symmetry breaking of the solutions for above equation as $q\nearrow 2$.

 Inspired by \cite{GZ2014}, our focuses here will be  on concentration behavior of nonlinear Hartree-type equation (\ref{eq1}) as  almost mass critical exponent $\gamma\nearrow 2$.
 Comparing with the results in \cite{GZ2014},  we should deal with the nonlocal term $(\frac{1}{|x|^\gamma}\ast |u|^2)u$, and the almost mass critical exponent comes from Hartree action $\frac{1}{|x|^\gamma}$, not the power exponent of $u$,  this will bring some
difficulties; on the other hand, in this case we can consider in  $\mathbb{R}^N$ for all $N\geq3$, while \cite{GZ2014} only  studied the cubic power exponent in $\mathbb{R}^2$.

 Firstly, we study the asymptotic behavior of  ground states of  equation (\ref{eq1}) for $V(x)\equiv 1, a=1$ and $\lambda=0$ as $\gamma\nearrow 2$, that is
 \begin{equation}\label{equ-gamma<2}
-\Delta u+u-(\frac{1}{|x|^\gamma}*|u|^2)u=0.
\end{equation}
We have the following Theorem.
 \begin{thm}\label{strong convergence theorem}
 Let $Q_{\gamma}=Q_{\gamma}(|x|)$ be a positive radial ground state for equation \eqref{equ-gamma<2} with $0<\gamma<2$. Then we have that
\begin{align}\notag
\|Q_{\gamma}\|_{L^2(\mathbb{R}^N)}\rightarrow\|Q_2\|_{L^2(\mathbb{R}^N)}\ \ \ \text{as $\gamma\nearrow 2$}.
\end{align}
Here $Q_2=Q_2(|x|)>0$ is a positive radial ground states for equation \eqref{equ-gamma=2}.
\end{thm}
 This is not an easy result, we shall employ the mountain pass structure of corresponding energy functional  to obtain the uniform bound for $\|Q_\gamma\|_{H^1}$,  and then obtain the  above convergence by Pohozaev identity.

Next we will show  asymptotic behaviors of  ground states of  equation (\ref{eq1}) for general $V(x)$. We consider a corresponding  minimizing variational  problem. Notice that, (\ref{eq1}) is an Euler-Lagrange equation of following minimizing problem:
\begin{align}\label{min}
e_a(\gamma)&=\inf_{\{u\in\mathcal{ H},\int_{\mathbb{R}^N}u^2=1\}}E_\gamma (u);\\
E_\gamma (u)&=\int_{\mathbb{R}^N}(|\nabla u|^2+V(x)|u|^2)
-\frac{a}{2}D_\gamma(u,u),\ \ D_\gamma(u,u):=\int_{\mathbb{R}^N\times\mathbb{R}^N} \frac{|u(x)|^2|u(y)|^2}{|x-y|^\gamma}dxdy.\notag
\end{align}
Where
$$\mathcal{H}:=\{u\in H^1(\mathbb{R}^N): \int_{\mathbb{R}^N} V(x)|u|^2dx<\infty\}.$$
Here we assume that $V(x):\mathbb{R}^N\rightarrow \mathbb{R}^+$ is locally bounded and satisfies $V(x)\rightarrow \infty$ as $|x|\rightarrow \infty$. Without loss
 of generality, by adding a suitable constant we may assume that
\begin{align*}
\inf _{x\in \mathbb{R}^N}V(x)=0,
\end{align*}
and $\inf_{x\in \mathbb{R}^N}V(x)$ can be attained. In this paper we are interested in addressing the limit behavior of minimizers for (\ref{min}) when $\gamma\nearrow2$ and $a>a^*$. By \cite{HZ2017,HZ2017-1} we have the following  Gagliardo-Nirenberg inequality:
\begin{align}\label{GN}
D_\gamma(u,u)
 \leq C_\gamma\big( \int_{\mathbb{R}^N}|\nabla u|^2 dx\big)^{\frac{\gamma}{2}}\big(\int_{\mathbb{R}^N}| u|^2 dx\big)^\frac{4-\gamma }{2},
\end{align}
where the best constant $C_\gamma=\frac{4}{4-\gamma}(\frac{4-\gamma}{\gamma})
^\frac{\gamma}{2}\frac{1}{\|Q_\gamma\|_2^2}$ and $N\geq3$. $Q_\gamma$ is an optimal minimizer of above inequality, satisfying equation (\ref{equ-gamma<2}), and also a positive radial ground state for (\ref{equ-gamma<2}).

\begin{thm}\label{concentration theorem}
Assume that $N\geq3$, $a>a^\ast=\|Q_2\|^2_2$ and $Q_2$ is the positive radially symmetric ground state of \eqref{equ-gamma=2}.  And also assume that
$$ V\in \mathcal{C}^1({\mathbb{R}^N}),~~\lim_{|x|\rightarrow \infty}V(x)=\infty~~~and~~\inf_{x\in \mathbb{R}^N}V(x)=0.$$
Let $u_\gamma\in\mathcal{H}$ be a non-negative minimizer of \eqref{min} with $\gamma\in (0,2)$. Then for each sequence $\{\gamma_k\}$ with $\gamma_k\nearrow2$ as $k\rightarrow\infty $, there exists a subsequence of $\{\gamma _k\}$, still denoted by $ \{\gamma _k\}$, such that $u_{\gamma _k}$ concentrates at a global minimum point $y_0$ of $V(x)$ in the following sense: for each large k, $u_{\gamma _k}$ has a unique global maximum point ${\bar{z}_k}\in \mathbb{R}^N$, satisfies
\begin{align}\label{a7}
\lim_{k\rightarrow \infty}\Big(\frac{a}{\|Q_{\gamma_k}\|_2^2}\Big)^{-\frac{2}{2-\gamma_k}}u_{\gamma_k}
\Big(\Big(\frac{a}{\|Q_{\gamma_k}\|_2^2}\Big)
^{-\frac{1}{2-\gamma_k }}~x+\bar{z}_k\Big)=\frac{1}{\|Q_2\|_2}Q_2(|x|)~\text{in}~ H^1(\mathbb{R}^N),
\end{align}
where $\bar{z}_k\rightarrow y_0$ as $k\rightarrow \infty$.
\end{thm}
\begin{remark}
Under the assumption of $V(x)$ in Theorem \ref{concentration theorem}, we have $\mathcal{H}\hookrightarrow L^p(\mathbb{R}^N)$ is compact, where $p\in[2,\frac{2N}{N-2})$. The  existence of the non-negative minimizer of \eqref{min} is similar to \cite{MS2013,Jean-Xia2017}.
\end{remark}
\begin{remark}
It follows from \cite{WY2017} that the uniqueness of positive ground state of \eqref{equ-gamma=2} holds for $N = 4$. If $N\geq3$ and $N\neq4$, we do not know that the positive ground state of equation \eqref{equ-gamma=2} is unique. Therefore, if $N=4$, we can obtain that $Q_\gamma\rightarrow Q_2$ strongly in $H^1(\mathbb{R}^4)$ and the right-hand side of the \eqref{a7} is unique. However, if $N\neq4$, we only know that there exists a positive radial ground state such that the above limit convergence to it.
\end{remark}

In the following we shall assume that the external potential $V$  satisfies that there exist $n\geq1$ distinct points $x_i\in \mathbb{R}^N$ with $V(x_i)=0$, while $V(x)>0$ otherwise. Moreover, there are numbers of $p_i>0 $ such that
\begin{align}\label{a8}
V(x)=O(|x-x_i|^{p_i})~near~x_i,~where~i=1,2,...,n,
\end{align}
and
$\lim_{x\rightarrow x_i}\frac{V(x)}{|x-x_i|^{p_i}}$ exists for all $1\leq i\leq n$.

Let $p=\max\{p_1,p_2,...,p_n\}$, and let $\lambda_i\in(0,\infty]$ be given by
\begin{align}\label{a9}
\lambda_i=\lim_{x\rightarrow x_i}\frac{V(x)}{|x-x_i|^{p_i}}.
\end{align}
Define $\lambda =\min\{\lambda_1,...,\lambda_n\}$ and let
\begin{align}\label{a10}
\mathcal{Z}:=\{x_i:\lambda_i=\lambda\}
\end{align}
denote the locations of the flattest global minima of $V(x)$. By the above notations, we have the following result, which tells us some further about
 the concentration point $y_0$ given by Theorem \ref{concentration theorem}.
\begin{thm}\label{theorem3}
Under the assumptions of Theorem \ref{concentration theorem} and let $V(x)$ satisfy also the additional condition (\ref{a8}), then the unique concentration point $y_0$ obtained in Theorem \ref{concentration theorem} has the properties:
\begin{align}\label{a11}
\lim_{k\rightarrow\infty}|\bar{z}_k-y_0|\Big(\frac{a}{\|Q_{\gamma_k}\|_2^2}\Big)
^{\frac{1}{2-\gamma_k}}=0~and~y_0\in\mathcal{Z}.
\end{align}
\end{thm}
This paper is organised as follows. Section 2 is devoted to the proof of Theorem \ref{strong convergence theorem} by using the mountain pass structure. In section 3, we prove Lemma \ref{energy relation} on energy estimates of minimizers for (\ref{min}). Finally, we use Lemma \ref{energy relation} to prove Theorem \ref{concentration theorem} and then utilize  the blowup analysis to complete the proof of Theorem \ref{theorem3} in section 4.
\section{The proof of Theorem \ref{strong convergence theorem}}
First note from \cite{MS2013} that the equation \eqref{equ-gamma<2} has a radially symmetric and monotone decreasing positive ground state solution $Q_\gamma$ which   satisfies the following Pohozaev identity:
\begin{align}\label{a4}
\frac{N-2}{2}\int_{\mathbb{R}^N}|\nabla Q_\gamma(x)|^2dx+\frac{N}{2}\int_{\mathbb{R}^N}|Q_\gamma(x)|^2dx
=\frac{2N-\gamma}{4}D_\gamma(Q_\gamma,Q_\gamma),
\end{align}
for all $0<\gamma<4$, one can derive from (\ref{equ-gamma<2}) and (\ref{a4}) that  $Q_\gamma(x)$ satisfies
\begin{align}\label{Pohazev-identity}
\frac{1}{\gamma}\int_{\mathbb{R}^N}|\nabla Q_\gamma(x)|^2dx=\frac{1}{4-\gamma}\int_{\mathbb{R}^N}|Q_\gamma(x)|^2dx=
\frac{1}{4}D_\gamma(Q_\gamma,Q_\gamma).
\end{align}
Now we define the energy functional of equation \eqref{equ-gamma=2} and \eqref{equ-gamma<2} by
\begin{align*}
J_\gamma(u)=\frac{1}{2}\int_{\mathbb{R}^N}|\nabla u|^2+|u|^2-\frac{1}{4}D_\gamma(u,u), \ \ \ \gamma\in(1,2].
\end{align*}
Let
\begin{align*}
G:=\{Q_2(x)\in H^1(\mathbb{R}^N):Q_2(x) \ \text{is a positive radial ground state of \eqref{equ-gamma=2}}\}.
\end{align*}
Combining \eqref{a4} and \eqref{Pohazev-identity}, we get the ground state energy
\begin{align*}
J_2(Q_2)=\frac{1}{2}\int_{\mathbb{R}^N}|Q_2(x)|^2dx,~\ \ ~Q_2\in G.
\end{align*}
Then we have $a^*=\int_{\mathbb{R}^N}|Q_2(x)|^2dx$. That is, all positive ground states of equation \eqref{equ-gamma=2} have the same $L^2$-norm. Similarly, we can obtain that all positive ground states of equation \eqref{equ-gamma<2} have the same $L^2$-norm if $\gamma$ is fixed.

To prove Theorem \ref{strong convergence theorem}, we divide the proof into several lemmas.
The ground state energy level $J_\gamma$ satisfies the mountain pass characterization, i.e.,
\begin{align}\notag
I_\gamma=\min_{u\in H^1(\mathbb{R}^N)\setminus\{0\}}\max_{t\geq0}J_\gamma(tu).
\end{align}
First, we have the following convergence holds.
\begin{lemma}\label{hartree-convergence}
Let $\{\gamma_n\}<2$ be a sequence converging to $2$. Then for any $u\in H^1_{rad}(\mathbb{R}^N)$
\begin{align*}
D_{\gamma_n}(u,u)\rightarrow D_2(u,u),~as~n\rightarrow\infty.
\end{align*}
Let $\{u_n\}\subset H^1_{rad}(\mathbb{R}^N)$ be a sequence converging weakly in $H^1_{rad}(\mathbb{R}^N)$ to some $u_0\in H^1_{rad}(\mathbb{R}^N)$. Then as $n\rightarrow\infty$
\begin{align*}
D_{\gamma_n}(u_n,u_n)\rightarrow D_2(u_0,u_0).
\end{align*}
\end{lemma}
\begin{proof}
By combining Hardy-Littlewood-Sobolev inequality and the compact Sobolev embedding. It is easy to see the above convergence holds.
\end{proof}
Next, we prove the following uniform estimate for the ground states $Q_\gamma$.
\begin{lemma}\label{bounded lemma}
Assume that $Q_\gamma\in H^1(\mathbb{R}^N)$ be the positive radial ground state of equation (\ref{equ-gamma<2}), where $\gamma\in(1,2)$. Then there exists a positive constant $C>0$ such that
\begin{align*}
\|Q_\gamma\|_{H^1}\leq C.
\end{align*}
\end{lemma}

\begin{proof}
We claim that
\begin{align}\notag
\limsup_{\gamma\rightarrow2}I_\gamma\leq I_2.
\end{align}
 Indeed, let $Q_2$ be a ground state of \eqref{equ-gamma=2}. Define $t_\gamma$ by the positive number satisfying
\begin{align}\notag
J_\gamma(t_\gamma Q_2)=\max_{t>0}J_\gamma(tQ_2).
\end{align}
Elementary computation shows that
\begin{align}\notag
t_\gamma=\Big(\frac{\|Q_2\|_{H^1}^2}{D_\gamma(Q_2,Q_2)}\Big)^{\frac{1}{2}}
\end{align}
and $\lim_{\gamma\rightarrow2}t_\gamma=1$. Now, by Lemma \ref{hartree-convergence}, we see from the mountain pass characterization of the ground states that as $\gamma\rightarrow2$,
\begin{align*}
J_\gamma(Q_\gamma)&\leq J_\gamma(t_\gamma Q_2)\\&=J_2(t_\gamma Q_2)+\frac{1}{4}\Big(\int_{\mathbb{R}^N}(\frac{1}{|x|^2}*|t_\gamma Q_2|^2)|t_\gamma Q_2|^2-\int_{\mathbb{R}^N}(\frac{1}{|x|^\gamma}*|t_\gamma Q_2|^2)|t_\gamma Q_2|^2\Big)\\&\leq J_2(Q_2)+o(1).
\end{align*}
Multiplying the equation \eqref{equ-gamma<2} by $Q_\gamma$ and integrating by part, we get
\begin{align}\notag
(\frac{1}{2}-\frac{1}{4})\|Q_\gamma\|_{H^1}^2=J_\gamma(Q_\gamma).
\end{align}
By the above argument, we know that $\|Q_\gamma\|_{H^1}$ is uniformly bounded for $0<\gamma<2$.
\end{proof}
The above Lemma \ref{bounded lemma} implies that $\{Q_{\gamma}\}$ is a bounded sequence in $H^1(\mathbb{R}^N)$. Therefore we can assume that, up to a subsequence, $Q_{\gamma_n}$ converges weakly to a nonnegative radial function $Q_\infty\in H^1_{rad}(\mathbb{R}^N)$, that is
\begin{align}\notag
Q_{\gamma_n}\rightharpoonup Q_\infty\ ~in~H^1(\mathbb{R}^N).
\end{align}
Moreover, by the compact embedding $H^1_{rad}(\mathbb{R}^N)\hookrightarrow L^q(\mathbb{R}^N)$ for any $2<q<\frac{2N}{N-2}$ (see Strauss \cite{WS1977}), we can assume that
\begin{align}\notag
Q_{\gamma_n}\rightarrow Q_\infty\ ~in~L^q(\mathbb{R}^N)
\end{align}
for any $2<q<\frac{2N}{N-2}$, and
\begin{align}\notag
Q_{\gamma_n}\rightarrow Q_\infty\ ~a.e.~in~\mathbb{R}^N.
\end{align}
By Lemma \ref{hartree-convergence}, we easily deduce that
\begin{align}\notag
\lim_{\gamma_n\rightarrow2}\int_{\mathbb{R}^N}(|x|^{-\gamma_n}*|Q_{\gamma_n}|^2)|Q_{\gamma_n}|^2dy
=\int_{\mathbb{R}^N}(|x|^{-2}*|Q_\infty|^2)|Q_\infty|^2dy.
\end{align}
Furthermore, we multiply the equation \eqref{equ-gamma<2} by $Q_\gamma$ and multiply the equation \eqref{equ-gamma=2} by $Q_\infty$ to get
\begin{align*}
\|Q_{\gamma_n}\|_{H^1}=D_{\gamma_n}(Q_{\gamma_n},Q_{\gamma_n})\rightarrow D_2(Q_\infty,Q_\infty)=\|Q_\infty\|_{H^1}.
\end{align*}
Combining this with the weak convergence of $Q_{\gamma_n}$, we obtain the strong convergence of $Q_{\gamma_n}$ to $Q_\infty$ in $H^1(\mathbb{R}^N)$.

\begin{lemma}\label{lower bounded lemma}
It satisfies
\begin{align}\notag
\liminf_{\gamma_n\rightarrow2}I_{\gamma_n}\geq I_2.
\end{align}
\end{lemma}
\begin{proof}
We see from the mountain pass characterization of ground states to \eqref{equ-gamma<2}, that for every $t>0$,
\begin{align*}
I_{\gamma_n}&\geq J_{\gamma_n}(tQ_{\gamma_n})\\
&=J_2(tQ_{\gamma_n})+\frac{t^4}{4}
\Big(\int_{\mathbb{R}^N}(\frac{1}{|x|^{2}}*| Q_{\gamma_n}|^2)| Q_{\gamma_n}|^2
-\int_{\mathbb{R}^N}(\frac{1}{|x|^{\gamma_n}}*|Q_{\gamma_n}|^2)|Q_{\gamma_n}|^2\Big)\\
&=J_2(tQ_{\gamma_n})+o(1)t^4
\end{align*}
as $\gamma_n\rightarrow2$. Taking $t=t_{\gamma_n}$ such that $t_{\gamma_n}$ satisfies
\begin{align}\notag
J_2(t_{\gamma_n} Q_{\gamma_n})=\max_{t\geq0}J_2(tQ_{\gamma_n}).
\end{align}
Elementary computation shows
\begin{align}\notag
t_{\gamma_n}=\Big(\frac{\|Q_{\gamma_n}\|_{H^1(\mathbb{R}^N)}^2}
{D_{\gamma_n}(Q_{\gamma_n},Q_{\gamma_n})}\Big)^{\frac{1}{2}}
\rightarrow\Big(\frac{\|Q_\infty\|_{H^1(\mathbb{R}^N)}^2}{D_2(Q_\infty,Q_\infty)}\Big)^{\frac{1}{2}}=1
~as~{\gamma_n}\rightarrow2
\end{align}
as in the proof of Lemma \ref{bounded lemma}. Then we get
\begin{align}\notag
I_{\gamma_n}\geq J_2(t_{\gamma_n} Q_{\gamma_n})+o(1)\geq I_2+o(1),
\end{align}
where the last inequality comes from the mountain pass characterization of $I_2$. The proof is complete.
\end{proof}
\renewcommand{\proofname}{\bf Proof of Theorem \ref{strong convergence theorem}}
\begin{proof}
Combining Lemma \ref{bounded lemma} and Lemma \ref{lower bounded lemma}, we conclude
\begin{align}\notag
\lim_{\gamma_n\rightarrow2}I_{\gamma_n}=I_2.
\end{align}
Hence we get
\begin{align}\notag
J_2(Q_\infty)=I_2.
\end{align}
In other words, $Q_\infty$ is a  positive radial ground state to \eqref{equ-gamma=2}, i.e. $Q_\infty\in G$ (defined above).

On the other hand, from the identity \eqref{Pohazev-identity}, we know that
\begin{align*}
I_\gamma=J_\gamma(Q_\gamma)=\frac{1}{4-\gamma}\int_{\mathbb{R}^N}|Q_\gamma(x)|^2dx~
\text{and}~
I_2=J_2(Q_2)=\frac{1}{2}\int_{\mathbb{R}^N}|Q_2(x)|^2dx.
\end{align*}
Therefore, from above we can obtain that
\begin{align}\notag
\int_{\mathbb{R}^N}|Q_\gamma(x)|^2dx\rightarrow\int_{\mathbb{R}^N}|Q_\infty(x)|^2dx
=a^*.
\end{align}
Let $\{Q_\gamma\}\subset H^1_{rad}(\mathbb{R}^N)$ be a family of positive radial ground state to \eqref{equ-gamma<2} for $\gamma$ near 2 and $\gamma<2$. Suppose $\|Q_\gamma\|_{L^2(\mathbb{R}^N)}$ does not converge to the $\|Q_2\|_{L^2(\mathbb{R}^N)}$. Then there exist a positive number $\varepsilon_0$ and a sequence $\{\gamma_k\}\rightarrow2$ such that $\|Q_{\gamma_k}\|_{L^2(\mathbb{R}^N)}-\|Q_2\|_{L^2(\mathbb{R}^N)}\geq\varepsilon_0$ which contradicts to Lemma \ref{bounded lemma} and Lemma \ref{lower bounded lemma}. The proof of Theorem \ref{strong convergence theorem} is complete.
\end{proof}
\renewcommand{\proofname}{Proof}
\section{Energy estimate}
In this section, the main purpose is to establish Lemma \ref{energy relation}. One can obtain from \cite{MS2013} that there exist positive constants $\delta$, $C$ and $R_0$ independent of $\gamma\in(0,2]$, such that $|x|>R_0$
\begin{align}\label{decay exponentially}
|Q_\gamma|+|\nabla Q_\gamma|\leq Ce^{-\delta|x|}~for~\gamma\in(0,2].
\end{align}
We next denote $\tilde{E}_\gamma(u)$ the following energy functional without the potential:
\begin{align}\label{without potential functional}
\tilde{E}_\gamma(u):=\int_{\mathbb{R}^N}|\nabla u(x)|^2dx-\frac{a}{2}D_\gamma(u,u)~u\in H^1(\mathbb{R}^N),
\end{align}
and consider the associated minimization problem
\begin{align*}
\tilde{e}_a(\gamma)=\inf_{\{u\in H^1(\mathbb{R}^N):\|u\|_2^2=1\}}\tilde{E}_\gamma(u).
\end{align*}
 Then the following Lemma gives refined information on the minimum energy $\tilde{e}_a(\gamma)$ as well as it minimizers.
\begin{lemma}\label{refine energy}
Let $ \gamma\in(0,2)$ and $ Q_\gamma$ be the radially symmetry positive solution of (\ref{equ-gamma<2}), then
\begin{align}\label{b3}
\tilde{e}_a(\gamma)= (1-\frac{2}{\gamma})(\frac{4-\gamma}{\gamma})
\Big(\frac{a}{\|Q_\gamma\|_2^2}\Big)^\frac{2}{2-\gamma},
\end{align}
and the  positive minimizers of $\tilde{e}_a(\gamma)$ must be of the form
\begin{align}\label{b4}
\tilde{Q}_\gamma(x)=\frac{\tau_\gamma^{N/2}}{\|Q_\gamma\|_2}Q_\gamma(\tau_\gamma x),~ where~ \tau_\gamma=\sqrt{\frac{\gamma}{4-\gamma}}
\Big(\frac{a}{\|Q_\gamma\|_2^2}\Big)^\frac{1}{2-\gamma}.
\end{align}
\end{lemma}
\begin{proof}
By using the Gagliardo-Nirenberg inequality (\ref{GN}), it follows from (\ref{without potential functional}) that
\begin{align}\notag
\tilde{E}_\gamma(u)\geq \int_{\mathbb{R}^N}|\nabla u|^2 dx-\frac{2a}{4-\gamma}\Big(\frac{4-\gamma}{\gamma}\Big)
^\frac{\gamma}{2}\frac{1}{\|Q_\gamma\|_2^2}\Big( \int_{\mathbb{R}^N}|\nabla u|^2 dx\Big)^{\frac{\gamma}{2}},
\end{align}
for any $u\in H^1(\mathbb{R}^N)$ and $\int_{\mathbb{R}^N}u^2=1 $.

Let
\begin{align}\label{b5}
g(s)=s-\frac{2a}{4-\gamma}\Big(\frac{4-\gamma}{\gamma}\Big)^\frac{\gamma}{2}
\frac{1}{\|Q_\gamma\|_2^2}s^{\frac{\gamma}{2}}~for~any~s\in[0,\infty).
\end{align}
We known that $g(s)$ attains its minimum at $s=(\frac{\gamma}{4-\gamma})(\frac{a}{\|Q_\gamma\|_2^2})^{\frac{2}{2-\gamma}}$, i.e. $s=\tau_\gamma ^2$,
which then implies that
\begin{align*}
\tilde{E}_\gamma(u)\geq g(\tau_\gamma ^2)=\Big(1-\frac{2}{\gamma}\Big)\Big(\frac{\gamma}{4-\gamma}\Big)
\Big(\frac{a}{\|Q_\gamma\|_2^2}\Big)^\frac{2}{2-\gamma}.
\end{align*}
This yields that
\begin{align}\label{b6}
\tilde{e}_a(\gamma)\geq g(\tau_\gamma ^2)=\Big(1-\frac{2}{\gamma}\Big)\Big(\frac{\gamma}{4-\gamma}\Big)
\Big(\frac{a}{\|Q_\gamma\|_2^2}\Big)^\frac{2}{2-\gamma}.
\end{align}

On the other hand, we introduce the following trial function
$$\psi_\gamma ^t(x)=\frac{t^{\frac{N}{2}}}{\|Q_\gamma \|_2}Q_\gamma (tx)$$
and $\int_{\mathbb{R}^N}|\psi_\gamma ^t|^2dx\equiv 1$ for all $t\in(0,\infty)$. We then obtain from (\ref{Pohazev-identity}) that
\begin{align*}
\tilde{e}_a(\gamma)\leq \tilde{E}_\gamma(\psi_\gamma ^t)
=\frac{\gamma}{4-\gamma}t^2-\frac{a}{2}\Big(\frac{4}{4-\gamma}\Big)
\frac{t^\gamma}{\|Q_\gamma\|_2^2},~for ~any ~t\in(0,\infty),
\end{align*}
Set $$h(t)=\frac{\gamma}{4-\gamma}t^2-
\frac{a}{2}\Big(\frac{4}{4-\gamma}\Big)\frac{t^\gamma}{\|Q_\gamma\|_2^2}.$$
We then obtain its minimum
$$h_{min}(t)=\Big(1-\frac{2}{\gamma}\Big)\Big(\frac{\gamma}{4-\gamma}\Big)
\Big(\frac{a}{\|Q_\gamma\|_2^2}\Big)^\frac{2}{2-\gamma}.$$
This and (\ref{b6}) then imply the estimate (\ref{b3}). Moreover, $\tilde{e}_a(\gamma)$ is attained at $\tilde{Q}_\gamma(x)
=\frac{\tau_\gamma^{N/2}}{\|Q_\gamma\|_2}Q_\gamma(\tau_\gamma x)$, and the proof is complete.
\end{proof}
\begin{remark}\label{remark}
For any fixed $a> a^\ast $, from the Theorem \ref{strong convergence theorem} we know that there exists a constant $\sigma > 1$, independent of $\gamma$, such that $\frac{a}{\|Q_\gamma\|_2^2}>\sigma>1$ as $\gamma$ is sufficiently close to $2^-$. Therefore, we further have
\begin{align}\label{b7}
\tau_\gamma=\sqrt{\frac{\gamma}{4-\gamma}}
\Big(\frac{a}{\|Q_\gamma\|_2^2}\Big)^\frac{1}{2-\gamma}\rightarrow +\infty
~and~\tilde{e}_a(\gamma)\rightarrow -\infty ~as~\gamma\nearrow 2.
\end{align}
\end{remark}
By applying Lemma \ref{refine energy}, we are able to establish the following estimates.
\begin{lemma}\label{energy relation}
 Let $a> a^\ast $ be fixed, and suppose that
$$V(x)\in L_{loc}^\infty(\mathbb{R}^N),~\lim_{|x|\rightarrow \infty}V(x)=\infty~and~ \inf_{x\in \mathbb{R}^N}V(x)=0.$$
Then
\begin{align}\label{b8}
e_a(\gamma)-\tilde{e}_a(\gamma)\rightarrow 0,~as~\gamma\nearrow 2.
\end{align}
Furthermore, we have
\begin{align}\label{b9}
\int_{\mathbb{R}^N}V(x)|u_\gamma(x)|^2dx\rightarrow 0 ~as~\gamma \nearrow 2,
\end{align}
and there exist two positive constants $C_1$ and  $C_2$ independent of $\gamma$, such that
\begin{align}\label{b15}
C_1\Big(\frac{a}{\|Q_\gamma\|_2^2}\Big)^\frac{2}{2-\gamma}&\leq \int_{\mathbb{R}^N}|\nabla u_\gamma (x)|^2dx
\leq C_2\Big(\frac{a}{\|Q_\gamma\|_2^2}\Big)^\frac{2}{2-\gamma}~as~\gamma\nearrow2,
\end{align}
\begin{align*}
C_1\Big(\frac{a}{\|Q_\gamma\|_2^2}\Big)^\frac{2}{2-\gamma}\leq& D_\gamma(u_\gamma,u_\gamma)
\leq C_2\Big(\frac{a}{\|Q_\gamma\|_2^2}\Big)^\frac{2}{2-\gamma}~as~\gamma \nearrow 2,
\end{align*}
where $u_\gamma(x)$ is a positive minimizer of (\ref{min}).
\end{lemma}
\begin{proof}
We choose a suitable trial function to estimate the upper bound of $e_a(\gamma)-\tilde{e}_a(\gamma)$.
For $R\geq 0$ fixed, let $\varphi_R(x)\in \mathcal{C}_0^\infty (\mathbb{R}^N)$ be a cut-off function such that
$\varphi_R(x)\equiv 1$ if $x\in B_R(0)$, $\varphi_R(x)\equiv 0$ if $x\in B_{2R}^c(0)$, and $0\leq \varphi_R(x)\leq 1$,
$\nabla \varphi (x)\leq \frac{C_0}{R}$ for any  $x\in B_{2R}(0)\setminus B_R(0)$. Set
\begin{align}\label{b10}
\omega_{R,\gamma}(x)=A_{R,\gamma}\tilde{\omega}_{R,\gamma}(x)=A_{R,\gamma} \varphi_R(x-x_0) \tilde{Q}_\gamma(x-x_0)~with~x_0\in\mathbb{R}^N,
\end{align}
where $\tilde{Q}_\gamma(x)$ defined in (\ref{b4}) is the positive minimizer of $\tilde{e}_a(\gamma)$, and $A_{R,\gamma}>0$ is chosen so that $\|\omega_{R,\gamma}\|_2^2=1$.
Now, we can calculate that
\begin{align*}
0\leq& e_a(\gamma)-\tilde{e}_a(\gamma)\leq E_\gamma (\omega_{R,\gamma}(x))-\tilde{e}_a(\gamma)\notag\\
=&E_\gamma( A_{R,\gamma}\tilde{\omega}_{R,\gamma}(x))
-\tilde{E}_\gamma(\tilde{\omega}_{R,\gamma}(x))+
\tilde{E}_\gamma(\tilde{\omega}_{R,\gamma}(x))
-\tilde{E}_\gamma(\tilde{Q}_\gamma(x))\notag\\
\leq& (A_{R,\gamma}^2-1)\int_{\mathbb{R}^N}|\nabla \tilde{\omega}_{R,\gamma}(x)|^2dx+
\frac{a}{2}(A_{R,\gamma}^4-1)D_\gamma(\tilde{\omega}_{R,\gamma}, \tilde{\omega}_{R,\gamma})\notag\\
&+\int_{\mathbb{R}^N}V(x)|\omega_{R,\gamma}(x)|^2dx+\Big|\int_{\mathbb{R}^N}|\nabla \tilde{Q}_\gamma(x)|^2dx-
\int_{\mathbb{R}^N}|\nabla \tilde{\omega}_{R,\gamma}(x)|^2dx\Big|\notag\\
&+ \frac{a}{2}\big|D_\gamma(\tilde{Q}_\gamma,\tilde{Q}_\gamma)
-D_\gamma(\tilde{\omega}_{R,\gamma},\tilde{\omega}_{R,\gamma})\big|\notag\\
=&A_1+A_2+A_3+A_4+A_5.
\end{align*}
By using \eqref{decay exponentially} and  $\tau_\gamma \rightarrow \infty$ as $\gamma\nearrow 2$, we then have
$$0\leq A_{R,\gamma}^2-1\leq \frac{\int_{B^c_{R_{\tau_\gamma}}} |Q_\gamma (x)|^2dx}{\int_{B_{R_{\tau_\gamma}}} |Q_\gamma (x)|^2dx}\leq
CR_{\tau_{\gamma}}e^{-2\delta R_{\tau_{\gamma}}}\leq Ce^{-\delta R_{\tau_\gamma}} ~as~\gamma\nearrow2,$$
where $\delta >0$ is as in \eqref{decay exponentially}. It hence follows from the above that
\begin{align*}
1\leq A_{R,\gamma}^4 \leq (1+ Ce^{-\delta R_{\tau_\gamma}})^2\leq 1+4 Ce^{-\delta R_{\tau_\gamma}}.
\end{align*}
Direct calculations show that
\begin{align*}
A_4&\leq \big| \int_{\mathbb{R}^N}|\nabla \tilde{Q}_\gamma(x)|^2dx-\int_{\mathbb{R}^N}(|\nabla \varphi_R|^2|\tilde{Q}_\gamma|^2+
|\varphi_R|^2|\nabla\tilde{Q}_\gamma|^2+2\nabla \varphi_R \varphi_R\nabla\tilde{Q}_\gamma \tilde{Q}_\gamma)dx\big|\notag\\
&\leq \int_{B_R^c}|\nabla \tilde{Q}_\gamma(x)|^2dx+\frac{C}{R^2}\int_{B_R^c}|\tilde{Q}_\gamma|^2dx
+\frac{2C}{R}\int_{B_R^c}|\nabla\tilde{Q}_\gamma|| \tilde{Q}_\gamma|dx
\leq Ce^{-\delta R_{\tau_\gamma}},
\end{align*}
where we use \eqref{decay exponentially}. One can also calculate that
\begin{align*}
A_5\leq Ce^{-\delta R_{\tau_\gamma}}.
\end{align*}
Moreover, we have
\begin{align}\notag
\lim_{\gamma\nearrow 2}\int_{\mathbb{R}^N}V(x)|\omega_{R,\gamma}|^2dx
=\lim_{\gamma\nearrow 2}\frac{A_{R,\gamma}^2}{\|Q_\gamma\|_2^2}\int V(\frac{x}{\tau_\gamma}+x_0)\varphi _R^2(\frac{x}{\tau_\gamma})Q_\gamma^2(x)dx=V(x_0)
\end{align}
holds for almost every $x_0\in \mathbb{R}^N$. Therefore, we choose $x_0\in \mathbb{R}^N$ such that $V(x_0)=0$, it follows from the above estimate that
\begin{align}\label{b14}
0\leq e_a(\gamma)-\tilde{e}_a(\gamma) \leq Ce^{-\delta R_{\tau_\gamma}}
+\int_{\mathbb{R}^N}V(x)|\omega_{R,\gamma}|^2dx\rightarrow0~as~\gamma\nearrow 2,
\end{align}
which then implies (\ref{b8}). Therefore, from \eqref{b8}, we can obtain \eqref{b9}. Nowadays the proof of \eqref{b15} is standard, and therefore we omit it.
\end{proof}

\section{Concentration  and symmetry breaking}
This section is devoted to proving Theorem \ref{concentration theorem} and Theorem \ref{theorem3} on the concentration and symmetry breaking of minimizers for (\ref{min}) as $\gamma \nearrow 2$, where $a>\|Q_2\|_2^2 $  is fixed. Set
\begin{align}\label{con-1}
\varepsilon _\gamma :=\varepsilon (\gamma)=\Big(\frac{a}{\|Q_\gamma\|_2^2}\Big)^{-\frac{1}{2-\gamma}}> 0,
\end{align}
then $\varepsilon _\gamma \rightarrow 0$ by Remark \ref{remark}. Define the $L^2(\mathbb{R}^N)$-normalized function
\begin{align}\label{con-2}
\tilde{v}_\gamma(x):=\varepsilon_\gamma^{\frac{N}{2}}u_\gamma(\varepsilon _\gamma x).
\end{align}
It then follows from Lemma \ref{energy relation} that there exist two positive constants $C_1$ and $C_2$, independent of $\gamma$, such that
\begin{align}\label{c3}
C_1\leq\int_{\mathbb{R}^N}|\nabla \tilde{v}_\gamma (x) |^2\leq C_2~~and~~
C_1\leq D_\gamma(\tilde{v}_\gamma,\tilde{v}_\gamma)\leq C_2.
\end{align}

By the above inequality \eqref{c3}, we easily obtain that
there exist a sequence $\{y_{\varepsilon _\gamma}\}$, $R_0>0$ and $\eta>0$ such that
\begin{align}\notag
\liminf_{\varepsilon _\gamma \rightarrow 0}\int_{B_{R_0}
(y_{\varepsilon _\gamma})}|\tilde{v}_\gamma (x) |^2dx\geq \eta > 0,
\end{align}
where $\tilde{v}_\gamma (x) $ is defined as (\ref{con-2}). For the sequence $\{y_{\varepsilon_\gamma}\}$ given by above, set
\begin{align}\label{c5}
v_\gamma (x)=\tilde{v}_\gamma(x+y_{\varepsilon _\gamma})
=\varepsilon_\gamma^{\frac{N}{2}}u_\gamma(\varepsilon_\gamma(x+y_{\varepsilon _\gamma})).
\end{align}
Then
\begin{align}\label{c6}
\liminf_{\varepsilon _\gamma \rightarrow 0}\int_{B_{R_0}(0)}|v_\gamma (x)|^2dx \geq \eta > 0,
\end{align}
which therefore implies that $v_\gamma(x)$ cannot vanish as $\gamma\nearrow2$.
We shall need the following technical result, proof of which can be found in \cite{GZ2014}.\vspace{-0.3cm}
\begin{lemma}\label{lemma4-1}
Assume $V(x)\in C^1(\mathbb{R}^N)$ satisfies $ \lim_{|x|\rightarrow \infty} V(x)=\infty$ and $\inf_{x\in \mathbb{R}^N}V(x)=0$. Then $\{\varepsilon _\gamma y_{\varepsilon _\gamma}\} $ is bounded uniformly for $\gamma \nearrow 2 $. Moreover, for any sequence $\{\gamma_k\}$ with $ \gamma_k \rightarrow 2 $ as $k\rightarrow \infty$, there exists a subsequence, still denoted by $\{\gamma_k\}$, such that
$z_k:= \varepsilon_k y_{\varepsilon_k}\rightarrow ^k y_0$, where $\varepsilon_k:=\varepsilon_{\gamma _k}$ is given by (\ref{con-1}), and
$y_0\in \mathbb{R}^N$ is a global minimum point of $V(x)$, i.e., $V(y_0)=0$.
\end{lemma}

Since $u_\gamma$ is a minimizer of (\ref{min}), it satisfies the Euler-Lagrange equation
\begin{equation}\label{equ-Eular-1}
-\Delta u_\gamma (x)+V(x)u_\gamma(x)=\mu_\gamma u_\gamma(x)+a(\frac{1}{|x|^\gamma}\ast |u_\gamma|^2)u_\gamma(x)~in~\mathbb{R}^N,
\end{equation}
where $\mu_\gamma \in \mathbb{R}$ is a Lagrange multiplier and satisfies
\begin{align}\notag
\mu _\gamma=e_a(\gamma)-\frac{a}{2}D_\gamma(u_\gamma,u_\gamma).
\end{align}
It then follows from Lemma \ref{energy relation} and (\ref{c3}) that there exist two positive constants $C_1$ and $C_2$, independent of $\gamma$, such that
\begin{align*}
-C_2< \mu_\gamma \varepsilon_\gamma ^2< -C_1~as~\gamma\nearrow 2 .
\end{align*}
By (\ref{con-1}) and (\ref{equ-Eular-1}), $v_\gamma(x)$ satisfies
\begin{equation}\label{equ-Eular-2}
-\Delta v_\gamma (x) +\varepsilon_\gamma^2V(\varepsilon_\gamma (x+y_{\varepsilon_\gamma}))v_\gamma (x)
=\varepsilon_\gamma^2\mu_\gamma v_\gamma (x)+\|Q_\gamma\|_2^2(\frac{1}{|x|^\gamma}*|v_\gamma|^2)v_\gamma (x)~in~\mathbb{R}^N.
\end{equation}
Therefore, by passing to a subsequence if necessary, we can assume that, for some number $\beta >0$,
$$\mu_{\gamma_k}\varepsilon_k^2\rightarrow-\beta^2<0~and~v_k:
=v_{\gamma_k}\rightharpoonup v_0\geq 0~in~H^1(\mathbb{R}^N) ~as~\gamma_k\nearrow 2,$$
for some $v_0 \in H^1(\mathbb{R}^N)$. By passing to the weak limit of (\ref{equ-Eular-2}), we deduce from Lemma \ref{energy relation} and Theorem \ref{strong convergence theorem} that the non-negative function $v_0$ satisfies
\begin{equation}\label{equ-Eular-limit}
-\Delta v_0(x)=-\beta^2v_0(x)+a^*(\frac{1}{|x|^2}*|v_0|^2)v_0 (x)~in ~\mathbb{R}^N.
\end{equation}
Furthermore, we infer from (\ref{c6}) that $v_0\not\equiv0$ in $\mathbb{R}^N $, and the strong maximum principle then yields that $v_0>0$ in $\mathbb{R}^N $. By the simple rescaling, we thus conclude from the positive  ground state of (\ref{equ-gamma=2}) that
\begin{align}\label{c8}
v_0=\frac{\beta^{\frac{N}{2}}}{\|Q_0\|_2}Q_0(\beta |x-x_0|)~for~some~x_0 \in \mathbb{R}^N,
\end{align}
where $Q_0$ is the positive radially symmetric solution of equation \eqref{equ-gamma=2}. Since $J_2(Q_0)\geq J_2(Q_2)=I_2$, where $I_2$ is the ground state energy, and $J_2(Q_0)=\frac{1}{2}|\|Q_0\|_2^2$, then we have $\|Q_0\|_2^2\geq a^*$. Therefore $\int_{\mathbb{R}^N}|v_0(x)|^2dx\geq1$. On the other hand, it follows from the Fatou's lemma that $\int_{\mathbb{R}^N}|v_0(x)|^2dx\leq1$. Then $\int_{\mathbb{R}^N}|v_0(x)|^2dx=1$, which implies that $\|Q_0\|_2=\|Q_2\|_2=a^*$. Thus, $Q_0$ is a positive radially symmetric ground state of equation \eqref{equ-gamma=2}.
Note that $\|v_k\|_2^2=1$, then $v_k$ converges to $v_0$ strongly in $L^2(\mathbb{R}^N)$ and in fact,
strongly in $L^p(\mathbb{R}^N)$ for any $2\leq p<\frac{2N}{N-2}$ because of $H^1(\mathbb{R}^N)$ boundedness. Furthermore, since  $v_k$ and  $v_0$ satisfy (\ref{equ-Eular-2}) and (\ref{equ-Eular-limit}) respectively, standard elliptic regularity theory gives that  $v_k$  converges to $v_0$ strongly in $H^1(\mathbb{R}^N)$.
\renewcommand{\proofname}{\bf Proof of Theorem \ref{concentration theorem}}
\begin{proof}
Motivated by \cite{GZ2014,WZ2017}, we are now ready to complete the proof of Theorem 1.1 by the following three steps.
\par $\mathbf{Step 1:}$ The decay property of $u_k:=u_{\gamma_k}$. For any sequence $\{\gamma_k\}$. Let $v_k:=v_{\gamma_k}\geq 0$ be defined by (\ref{c5}).
The above analysis shows that there exists a subsequence, still dented by $\{v_k\}$, satisfying (\ref{equ-Eular-2}) and $v_k\rightarrow v_0$ strongly in $H^1(\mathbb{R}^N)$ for some positive function $v_0$. Hence for any $2<\alpha<\frac{2N}{N-2}$,
\begin{align*}
\int_{|x|\geq R}|v_k|^\alpha dx\rightarrow0~\text{as $R\rightarrow\infty$ uniformly for large $k$}.
\end{align*}
Since $\mu_{\gamma _k}<0$ , it follows from (\ref{equ-Eular-2}) that
$$-\Delta v_k-c(x)v_k \leq 0 ,~where~c(x)=\|Q_\gamma \|_2^2(\frac{1}{|x|^\gamma }*|v_k |^2).$$
Denote $\phi_{v_k}(x)=\int_{\mathbb{R}^N}\frac{|v_k(y)|^2}{|x-y|^\gamma}dy$. By the Riesz potential inequality, we then have
\begin{align*}
\|\phi_{v_k}(x)\|_{L^q(B_2(\xi))}\leq \|\phi_{v_k}(x)\|_{L^q(\mathbb{R}^N)} \leq C\|v_k^2\|_{L^p(\mathbb{R}^N)}=C\|v_k\|^2_{L^{2p}(\mathbb{R}^N)}
\end{align*}
where $1+\frac{1}{q}=\frac{1}{p}+\frac{\gamma}{N}$. In particular, if $q=\frac{N}{\gamma}>\frac{N}{2}$, then $p=1$. Since $v_k\in H^1(\mathbb{R}^N)$, by the Sobolev embedding theorem, we get
\begin{align*}
\|\phi_{v_k}(x)\|_{L^q(B_2(\xi))}\leq\|v_k\|^2_{L^{2p}(\mathbb{R}^N)}<C<\infty.
\end{align*}
\par Note that $q>\frac{N}{2}$, by applying De Giorgi-Nash-Moser theory (see \cite{HL2011}, Theorem 4.1), we thus have
\begin{align}\label{c10}
\max_{B_1(\xi)}v_k(x) \leq C (\int_{B_2(\xi)}|v_k(x)|^\alpha dx)^{\frac{1}{\alpha}},
\end{align}
where $\xi$ is an arbitrary point in $\mathbb{R}^N $, and $C$ is a constant depending only on the bound of $\|\phi_{v_k}(x)\|_{L^p(B_2(\xi))}$. Hence we deduce from (\ref{c10}) that
\begin{align*}
v_k(x)\rightarrow0~\text{as~$|x|\rightarrow\infty$~uniformly~in~$k$}.
\end{align*}
Since $v_k$ satisfies (\ref{equ-Eular-2}), one can use the comparison principle as in \cite{KW1994} to $v_k$ with $Ce^{-\frac{\beta}{2}|x|}$, which then shows that
there exists a large constant $R>0$, independent of $k$, such that
\begin{align}\label{c11}
v_k(x)\leq Ce^{-\frac{\beta}{2}|x|}~for~|x|>R~as~~k\rightarrow\infty.
\end{align}
By Lemma \ref{lemma4-1}, we therefore obtain from (\ref{c11}) that the subsequence
\begin{align}\notag
u_k(x):=u_{\gamma_k}(x) =\frac{1}{\varepsilon_k^{\frac{N}{2}}}v_k\Big(\frac{x-z_k}{\varepsilon_k}\Big),
\end{align}
decays uniformly to zero for $x$ outside any fixed neighborhood of $y_0$ as $k\rightarrow \infty$, where $\varepsilon_k=\varepsilon_{\gamma_k}$,
$z_k\in \mathbb{R}^N$ is defined as in Lemma \ref{lemma4-1}, and $y_0\in \mathbb{R}^N $ is a global minimum point of $V(x)$.
\par $\mathbf{Step 2:}$ The detailed concentration behavior. Let $\bar{z}_k$ be any local maximum point of $u_k$. By the definition of $\phi_{u_k}(\bar{z}_k)$, we have
\begin{align*}
\phi_{u_k}(\bar{z}_k)=\int_{\mathbb{R}^N}\frac{|u_k(y)|^2}{|\bar{z}_k-y|^\gamma}dy
\leq u_k(\bar{z}_k)^{\frac{1}{N}}
(\int_{|\bar{z}_k-y|<\delta}\frac{|u_k(y)|^{2-\frac{1}{N}}}{|\bar{z}_k-y|^\gamma}dy
+\int_{|\bar{z}_k-y|\geq\delta}\frac{|u_k(y)|^{2-\frac{1}{N}}}{|\bar{z}_k-y|^\gamma}dy\Big).
\end{align*}
Since $0<\gamma<2$, then it follows from the H\"{o}lder inequality that
\begin{align*}
\int_{|\bar{z}_k-y|<\delta}\frac{|u_k(y)|^{2-\frac{1}{N}}}{|\bar{z}_k-y|^\gamma}dy
&\leq\big(\int_{|\bar{z}_k-y|<\delta}|\bar{z}_k-y|^{-\frac{\gamma N}{2}}dy\big)^{\frac{2}{N}}
\big(\int_{|\bar{z}_k-y|<\delta}|u_k(y)|^{2+\frac{3}{N-2}}\big)^{\frac{N-2}{N}}\\
&\leq C<\infty,
\end{align*}
where $C>0$ is independent of $k$, since $u_k$ is uniformly bounded in $H^1(\mathbb{R}^N)$.

On the other hand, combining H\"{o}lder inequality and Sobolev inequality yields that
\begin{align*}
\int_{|\bar{z}_k-y|\geq\delta}\frac{|u_k(y)|^{2-\frac{1}{N}}}{|\bar{z}_k-y|^\gamma}&\leq
\big(\int_{|\bar{z}_k-y|\geq\delta}|\bar{z}_k-y|^{-\frac{(N+1)\gamma}{\gamma}}dy\big)^{\frac{\gamma}{N+1}}
\big(\int_{|\bar{z}_k-y|\geq\delta}|u_k(y)|^{2+\frac{2N\gamma}{N(N+1-\gamma)}}dy\big)
^{\frac{N+1-\gamma}{N+1}}\\&\leq C<\infty,
\end{align*}
where $C>0$ is independent of $k$. According to the above two estimates, we deduce that
\begin{align}\label{c12}
 \phi_{u_k}(\bar{z}_k)\leq C u_k({\bar{z}_k})^{\frac{1}{N}}~as~\gamma\nearrow 2.
\end{align}
We then follows from (\ref{equ-Eular-1}) and (\ref{c12}) that
\begin{align*}
0\leq\mu_\gamma u_\gamma(\bar{z}_k)+\|Q_\gamma\|^2_2\phi_{u_k}(\bar{z}_k)u_\gamma(\bar{z}_k)
\leq\mu_\gamma u_\gamma(\bar{z}_k)+\|Q_\gamma\|^2_2C u_\gamma^{\frac{N+1}{N}}(\bar{z}_k)
~as~\gamma\nearrow2,
\end{align*}
which implies that
\begin{align*}
u_\gamma(\bar{z}_k)\geq\Big(\frac{-\mu_\gamma}{\|Q_\gamma\|^2_2C}\Big)^N\geq C\varepsilon_\gamma^{-2N}.
\end{align*}
This estimate and the above decay property thus imply that $\bar{z}_k\rightarrow y_0$ as $k\rightarrow \infty$. Set
\begin{align}\label{c13}
\tilde{v}_k=\varepsilon_k^{\frac{N}{2}}u_k(\epsilon_kx+\bar{z}_k),
\end{align}
so that $\tilde{v}_k$ satisfies (\ref{c3}). It then follow from (\ref{equ-Eular-2}) that
\begin{align}\label{eq6}
-\Delta \tilde{v} _k (x) +\varepsilon_k^2V(\varepsilon_k x+\bar{z}_k)\tilde{v}_k(x)
=\varepsilon_{\gamma_k}^2\mu_{\gamma_k }\tilde{v}_k(x)
+\|Q_\gamma\|_2^2(\frac{1}{|x|^\gamma}\ast|\tilde{v}_k|^2)\tilde{v}_k (x)~in~\mathbb{R}^N.
\end{align}
The same argument as proving (\ref{equ-Eular-limit}) yields that there exists a subsequence of $\{\tilde{v}_k\}$, still denoted by $\{\tilde{v}_k\}$,
such that $\tilde{v}_k\rightarrow \tilde{v}_0$ as $k\rightarrow\infty$ in $H^1(\mathbb{R}^N)$ for some nonnegative function
$\tilde{v}_0\geq0$, where $\tilde{v}_0$ satisfies (\ref{equ-Eular-limit}) for some constant $\beta>0$. We derive from (\ref{eq6}) that
\begin{align*}
\tilde{v}_k(0)\geq\Big(\frac{-\varepsilon_{\gamma_k}^2\mu_{\gamma_k }}{C\|Q_\gamma \|_2^2}\Big)^N\geq
\Big(\frac{\beta^2}{C\|Q_\gamma \|_2^2}\Big)^N~as~k\rightarrow \infty,
\end{align*}
which implies that $\tilde{v}_0(0)>0$. Thus, the strong maximum principle yields that $\tilde{v}_0(x)>0$ in $\mathbb{R}^N$. Since $x=0$ is a critical point of $\tilde{v}_k$ for all $k>0$, it is also a critical point of $\tilde{v}_0$. We therefore conclude from the positive radial solutions of equation (\ref{equ-gamma=2}) that $\tilde{v}_0$ is spherically symmetric about the origin, and
\begin{align}\label{c14}
\tilde{v}_0=\frac{\beta^{\frac{N}{2}}}{\|Q_2\|_2}Q_2(\beta|x|)~for~some~\beta>0.
\end{align}
$\mathbf{Step 3:}$ The exact value of $\beta$ defined in (\ref{c14}). Let $\{\gamma_k\}$, where $\gamma_k\nearrow 2$ as $k\rightarrow \infty$, be the subsequence obtained in Step 2, and denoted $u_k:=u_{\gamma_k}$. Recall from Lemma \ref{energy relation} that
\begin{align*}
e_a(\gamma_k)&=\tilde{e}_a(\gamma_k)+o(1)
=\Big(1-\frac{2}{\gamma_k}\Big)\Big(\frac{4-\gamma_k}{\gamma_k}\Big) \varepsilon_k^{-2}+o(1)~as~k\rightarrow\infty,
\end{align*}
which yields that
\begin{align}\label{c15}
\lim_{k\rightarrow \infty}\frac{\gamma_k}{2-\gamma_k}\varepsilon_k^2e_a(\gamma_k)
=-\lim_{k\rightarrow \infty}\frac{4-\gamma_k}{\gamma_k} =-1.
\end{align}
 On the other hand,
\begin{align}\label{c16}
e_a(\gamma_k)&=\int_{\mathbb{R}^N}(|\nabla u_{\gamma_k}|^2+V(x)|u_{\gamma_k}|^2)dx
-\frac{a}{2}D_{\gamma_k}(u_{\gamma_k},u_{\gamma_k})\notag\\
&=\varepsilon_k^{-2}\Big(\int_{\mathbb{R}^N}|\nabla \tilde{v}_k|^2dx
-\frac{\|Q_{\gamma_k}\|_2^2}{2}D_{\gamma_k}( \tilde{v}_k, \tilde{v}_k)\Big)
+\int_{\mathbb{R}^N}V(x)| u_{\gamma_k}|^2dx\notag\\
&\geq \varepsilon_k^{-2}\Big(\int_{\mathbb{R}^N}|\nabla \tilde{v}_k|^2dx
-\frac{2}{4-\gamma_k}(\frac{4-\gamma_k}{\gamma_k})^{\frac{\gamma_k}{2}}
(\int_{\mathbb{R}^N}|\nabla \tilde{v}_k|^2dx)^{\frac{\gamma_k}{2}}\Big)
\end{align}
where $\tilde{v}_k:= \tilde{v}_{\gamma_k}$ is as in (\ref{c13}). Set $\beta_{\gamma_k}^2:=\int_{\mathbb{R}^N}|\nabla \tilde{v}_k|^2dx$.
Since $\tilde{v}_k(x)\rightarrow \tilde{v}_0(x)$ strongly in $H^1(\mathbb{R}^N)$, we have
\begin{align}\label{c17}
\lim_{k\rightarrow\infty}\beta_{\gamma_k}^2=\|\nabla\tilde{v}_0\|_2^2=\beta^2,
\end{align}
where (\ref{a4}) is used. Let $f_k(t)=t-\frac{2}{4-\gamma_k}(\frac{4-\gamma_k}{\gamma_k})
^\frac{\gamma_k}{2}t^{\frac{\gamma_k}{2}}$, where $t\in (0,\infty)$. A simple analysis shows that $f_k(\cdot)$ attains its global minimum at the unique point $t_k=\frac{\gamma_k}{4-\gamma_k}$ and also
$f_k(t_k)=\frac{\gamma_k-2}{4-\gamma_k}$. We hence deduce from (\ref{c16}) that
\begin{align*}
\lim_{k\rightarrow \infty}\frac{\gamma_k}{2-\gamma_k}\varepsilon_k^2e_a(\gamma_k)
\geq\lim_{k\rightarrow \infty}\frac{\gamma_k}{2-\gamma_k}f_k(\beta_{\gamma_k}^2)\geq \lim_{k\rightarrow \infty}\frac{\gamma_k}{2-\gamma_k}f_k(t_k) =-1,
\end{align*}
 Combine with (\ref{c15}), it follows that
$$\lim_{k\rightarrow\infty}\frac{f_k(\beta_{\gamma_k}^2)}{f_k({t_k)}}=1.$$
We then obtain that
\begin{align*}
\lim_{k\rightarrow\infty}\beta_{\gamma_k}^2=\lim_{k\rightarrow\infty}t_k=1,
\end{align*}
and therefore we have $\beta=1$ by applying (\ref{c17}), which, together with (\ref{c13}) and (\ref{c14}) give in (\ref{a7}). We thus complete the proof of Theorem \ref{concentration theorem}.
\end{proof}
Following the proof of Theorem \ref{concentration theorem}, we next address Theorem \ref{theorem3} on the local properties of concentration points. Under the assumption (\ref{a8}), we first denoted $$\bar{V}_i(x)=\frac{V(x)}{|x-x_i|^{p_i}},~where~i=1,2,...,n,$$
so that the $\lim_{x\rightarrow _i}\bar{V}_i(x)=V_i(x_i)$ is assumed to exist for all $i=1,2,...,n.$
\renewcommand{\proofname}{\bf Proof of Theorem \ref{theorem3}}
\begin{proof}
For convenience we still denoted $\{\gamma_k\}$ to be the subsequence  obtained in Theorem \ref{concentration theorem}. Choose a point $x_{i_0}\in \mathcal{Z}$, where $\mathcal{Z}$ is defined by (\ref{a9}), and let
$$\omega_{R,\gamma_k}(x)=A_{R,\gamma_k} \varphi_R(x-x_{i_0}) \tilde{Q}_{\gamma_k}(x-x_{i_0})$$
be the trial function defined by (\ref{b10}). By (\ref{b14}), we know that
\begin{align}\label{c18}
&e_a(\gamma_k)-\tilde{e}_a(\gamma_k)\notag\\
&\leq \int_{\mathbb{R}^N}V(x)|A_{R,\gamma_k}(x)|^2dx+Ce^{-\delta R_{\tau_{\gamma_k}}}\notag\\
&\leq\frac{A_{R,\gamma_k}^2}{\tau_{\gamma_k}^p\|Q_{\gamma_k}\|_2^2}
\int_{B_{2R_{\tau_{\gamma_k}}}}\bar{V}_{i_0}(\frac{x}{\tau_{\gamma_k}+x_{i_0}})
|x|^pQ_{\gamma_k}^2(x)dx+Ce^{-\delta R_{\tau_{\gamma_k}}}\notag\\
&\leq\frac{A_{R,\gamma_k}^2}{\tau_{\gamma_k}^p\|Q_{\gamma_k}\|_2^2}
\int_{\mathbb{R}^N}\mathcal{X}_{B_{2R_{\tau_{\gamma_k}}}}\bar{V}_{i_0}
\big(\frac{x}{\tau_{\gamma_k}+x_{i_0}}\big)|x|^pQ_{\gamma_k}^2(x)dx+Ce^{-\delta R_{\tau_{\gamma_k}}}.
\end{align}
where $\tau_{\gamma_k} >0$ satisfies $\tau_{\gamma_k}=\sqrt{\frac{\gamma}{4-\gamma}}\frac{1}{\varepsilon_k}$
in view of Lemma \ref{refine energy} and (\ref{con-1}), and $\mathcal{X}_{R_{2R_{\tau_{\gamma_k}}}}$ is the characteristic function of the set $B_{2R_{\tau_{\gamma_k}}}$. Combining Theorem \ref{strong convergence theorem} and \eqref{decay exponentially}, we deduce that
\begin{align*}
\mathcal{X}_{B{2R_{\tau_{\gamma_k}}}}\bar{V}_{i_0}(\frac{x}{\tau_{\gamma_k}}
+x_{i_0})|x|^pQ_{\gamma_k}^2(x)\leq \sup_{B_{2R}}\bar{V}_{i_0}(x+x_{i_0})
\cdot Ce^{-\delta |x|}\in L^1(\mathbb{R}^N),
\end{align*}
and
\begin{align*}
\mathcal{X}_{B{2R_{\tau_{\gamma_k}}}}\bar{V}_{i_0}
(\frac{x}{\tau_{\gamma_k}}+x_{i_0})|x|^pQ_{\gamma_k}^2(x)\rightarrow \bar{V}_{i_0}(x_{i_0})|x|^pQ_2^2(x)
~a.e.~\mathbb{R}^N~as~k\rightarrow\infty.
\end{align*}
Note that $A_{R,\gamma_k}\rightarrow 1$ as $\gamma_k\nearrow 2$, we thus obtain from (\ref{c18}) and Lebesgue's dominated convergence theorem that
\begin{align}\label{c19}
\lim_{k\rightarrow\infty}\frac{e_a(\gamma_k)-\tilde{e}_a(\gamma_k)}{\varepsilon_k^p}
\leq\frac{\bar{V}_{i_0}(x_{i_0})}{\|Q_2\|_2^2}\int_{\mathbb{R}^N}|x|^pQ_2^2dx.
\end{align}

On the other hand, following the proof of Theorem \ref{concentration theorem} we denote $\bar{z}_k$ to be the unique global maximum point of $u_k$, and let $\tilde{v}_k$ be defined as in (\ref{c13}). Denote also $y_0\in \mathbb{R}^N$  to be the limit of $\bar{z}_k$ as $k\rightarrow \infty$. Since $V(y_0)=0$, then there exists an $x_j=y_0$ for some $1\leq j\leq n$. We claim that $\{\frac{\bar{z}_k-x_j}{\varepsilon_k}\}$ is bounded in $\mathbb{R}^N$. Indeed, if there  exists a subsequence, still denoted by $\{\gamma_k\}$, such that $|\frac{\bar{z}_k-x_j}{\varepsilon_k}|\rightarrow \infty$ as $k\rightarrow \infty $, it then follows from Fatou's Lemma that, for any $C>0$ sufficiently large,
\begin{align*}
\lim_{k\rightarrow\infty}\frac{e_a(\gamma_k)-\tilde{e}_a(\gamma_k)}{\varepsilon_k^{p_i}}
&\geq \lim_{k\rightarrow\infty}\int_{\mathbb{R}^N}\bar{V}_j(\varepsilon_k+\bar{z}_k)\Big|x
+\frac{\bar{z}_k-x_j}{\varepsilon_k}\Big|^{p_j}\tilde{v}_k^2dx\notag\\
&\geq\int_{\mathbb{R}^N}\lim_{k\rightarrow\infty}\bar{V}_j(\varepsilon_k+\bar{z}_k)\Big|x
+\frac{\bar{z}_k-x_j}{\varepsilon_k}\Big|^{p_j}\tilde{v}_k^2dx
\geq C\bar{V}_j(x_j),
\end{align*}
which however contradicts (\ref{c19}) owing to $p_j\leq p=max\{p_1,p_2,...,p_n\}$, and the claim is therefore true. Consequently, there exists a subsequence, still denoted by $\{\gamma_k\}$, such that
\begin{align*}
\frac{\bar{z}_k-x_j}{\varepsilon_k}\rightarrow \bar{z}_0~for~some~\bar{z}_0\in\mathbb{R}^N.
\end{align*}
Since $Q_2$ is a radial decreasing function and decays exponentially as $|x|\rightarrow \infty$, we then deduce that
\begin{align}\label{c20}
\lim_{k\rightarrow\infty}\frac{e_a(\gamma_k)-\tilde{e}_a(\gamma_k)}{\varepsilon_k^{p_i}}
&\geq \lim_{k\rightarrow\infty}\int_{\mathbb{R}^N}\bar{V}_j(\varepsilon_k+\bar{z}_k)\Big|x
+\frac{\bar{z}_k-x_j}{\varepsilon_k}\Big|^{p_j}\tilde{v}_k^2dx\notag\\
&\geq\bar{ V}_j(x_j)\int_{\mathbb{R}^N}|x+\bar{z}_0|^{p_j}\tilde{v}_0^2dx\notag\\
&\geq\frac{\bar{V}_j(x_j)}{\|Q_2\|_2^2}\int_{\mathbb{R}^N}|x|^{p_j}Q_2^2dx.
\end{align}
where $\tilde{v}_0>0$ is as in (\ref{c14}), and $``="$ in the last inequality of (\ref{c20}) holds if and only if $\bar{z}_0=0$.

Applying (\ref{c19}) and (\ref{c20}), it is not difficult to see that $p_j=p$ and then $V_j(x_j)=V_{i_0}(x_{i_0})$. Hence, $x_j=y_0\in \mathcal{Z}$ must be the flattest global minimum point of $V(x)$. Based on these facts, using (\ref{c19}) and (\ref{c20}) we see that (\ref{c20}) is essentially an equality. Therefore, $\bar{z}_0=0$ and (\ref{a11}) holds. The proof of Theorem \ref{theorem3} is completed.
\end{proof}

\end{document}